\documentclass{article}
\usepackage{cmap}
\usepackage{amsthm}
\usepackage[cp1251]{inputenc}
\usepackage[T2A]{fontenc}
\usepackage{amsfonts}
\usepackage{amssymb}
\usepackage{amsmath}
\usepackage{color}
\usepackage{hyperref}
\usepackage{mathdots}
\usepackage{cite}
\usepackage[left=2.5cm,right=2.5cm, top=2.5cm,bottom=3cm]{geometry}
\newtheorem{lemma}{Lemma}
\newtheorem{prop}{Proposition}
\newtheorem*{prop*}{Proposition}

\newtheorem{theorem}{Theorem}
\newtheorem*{theorem*}{Theorem}
\newtheorem*{Shmelkintheorem*}{Shmel'kin's theorem}
\theoremstyle{definition}
\newtheorem*{defin}{Definition}
\newtheorem{question}{Question}

\theoremstyle{remark}
\newtheorem*{rem}{Remark}

\let\tilde\widetilde
\let\hat\widehat

\begin{document}

\title{Infinite systems of equations
in abelian and nilpotent groups\footnote{
This work was supported by
the Russian Science Foundation, 
project no. 22-11-00075}
}
\author{Mikhail A. Mikheenko
\\{\small
Faculty of Mechanics and Mathematics
of Lomonosov Moscow State University
}
\\{\small
Moscow Center for
Fundamental and Applied Mathematics}\\
{\normalsize mamikheenko@mail.ru}
}
\date{}
\maketitle

\begin{abstract}
Every abelian (and even every nilpotent)
group contains a solution of
any finite unimodular system
of equations over itself.
However, this is not true for infinite systems.
We deduced a criterion
for a periodic abelian group
to contain a solution of
any infinite unimodular system
of equations over itself.
Using this criterion, we show
that nilpotent groups of bounded period
also contain solutions of
all infinite unimodular systems of equations
over themselves.
Solvability of every nonsingular
infinite system of equations
in each divisible nilpotent group
is shown as well.
\end{abstract}

\section{Introduction}

In this work $\mathbb Z_p$ for a prime number $p$ denotes
the field of order $p$. Here we refer to the exponent of a group $G$
as the period of $G$.

The present article is devoted to equations over groups.

\begin{defin}
Let $G$ be a group.
An equation in variables $x_1,\ldots,x_n$
over $G$ is an expression (possibly with coefficients from $G$)
$w(x_1,\ldots,x_n)=1$, where $w$ is an element of the free product
$G*F(x_1,\ldots,x_n)$,
in which $F(x_1,\ldots,x_n)$ is the free group
with basis $x_1,\ldots,x_n$.

The equation $w(x_1,\ldots,x_n)=1$ is {\it solvable} in the group
$\tilde G$ if $\tilde G \supset G$ and $\tilde G$ contains a
solution of this equation (i.e. there are elements
$\tilde g_1,\ldots,\tilde g_n \in \tilde G$ such that $w(\tilde g_1,\ldots, \tilde g_n)=1$).
The group $\tilde G$ is called a {\it solution group}.
Equivalently, $w=1$ is solvable in $\tilde G$
if there is a homomorphism $G*F(x_1,\ldots,x_n) \to \tilde G$
which is injective on $G$ and sends $w$ to $1$.

If $w=1$ is solvable in some group $\tilde G$,
we say that the equation $w=1$ over $G$ is {\it solvable}.

The solvability of a (finite or an infinite) system of equations
(possibly in an infinite set of variables)
over a group is defined likewise.
\end{defin}

The study of equations over groups has a long history
but continues to this day. For example, see
\cite{B80, B84, K06, KP95, M24, Sh67, ABA21, BE18, EH91, EH21, EdJu00, 
GR62, How81, IK00, K93, KM23, KMR24, KT17, G83, Le62, NT22, P08, Sh81, T18}.
There is a survey on equations over groups as well, see \cite{Ro12}.

Here we are interested in nonsingular systems of equations.

\begin{defin}
Let $\{w_1=1,\ldots,w_k=1\}$ be
a finite system of equations in variables
$\{x_1,\ldots,x_n\}$
over a group
$G$.
Take the integer matrix
$(a_{ij})$
of size $k \times n$,
where $a_{ij}$ is
the exponent sum of the variable
$x_j$ in the word $w_i$
(we also call this number
an exponent sum of the variable $x_j$
in the equation $w_i=1$).
For example, the exponent sum of $x$
in the equation
$x^{-1}gyx^2g^2x^{-3}y^{2}=1$
equals $-2$.

The system $\{w_1=1,\ldots,w_k=1\}$
is called {\it nonsingular}
if the rows of the matrix $(a_{ij})$
are linearly independent over $\mathbb Q$.
Equivalently, if these rows
are independent as elements of a $\mathbb Z$-module,
this is if no combination of these rows
is equal to zero, apart from the 
combination where every coefficient is zero.

Let $p$ be a prime number.
Let $\tilde a_{ij}$ be a residue
class of $a_{ij}$ modulo $p$.
The system $\{w_1=1,\ldots, w_k=1\}$
is called {\it $p$-nonsingular}
if the rows of the matrix $(\tilde a_{ij})$
are linearly independent over $\mathbb Z_p$.

Let $\pi$ be some set of prime numbers (maybe not of all prime numbers).
A finite nonsingular system of equations is called {\it $\pi$-nonsingular}
if it is $p$-nonsingular for every prime number $p\in\pi$.
Patricularly, $\varnothing$-nonsingularity of a system
is the same as its nonsingularity.

A finite system of equations is called {\it unimodular}
if it is $\Pi$-nonsingular, where $\Pi$ is a set of all prime numbers.

An infinite system of equations
is called {\it nonsingular}
({\it $p$-nonsingular}, {\it $\pi$-nonsingular} respectively),
if its every finite subsystem is nonsingular
($p$-nonsingular, $\pi$-nonsingular respectively).
\end{defin}

In particular, a single equation $w(x)=1$ in
a single variable is nonsingular
($p$-nonsingular, unimodular respectively)
if the exponent sum of $x$ in $w=1$
is not equal to zero
(is not divisible by $p$, is equal to $\pm 1$ respectively).

One can consider an integer matrix
of exponent sums of variable for
an infinite system $\{w_j = 1\}_{j \in J}$ of equations
in variables $\{x_i\}_{i \in I}$ over a group $G$ as well.
It is a matrix of size $I \times J$
with finitely supported rows,
i.e. rows which have only a finite number
of non-zero elements.
Clearly, a system of equations is nonsingular
if and only if the rows of the corresponding matrix
are linearly independent over $\mathbb Q$.
Analogously, the system is $p$-nonsingular
if and only if the rows containing the corresponding
residue classes of the elements of the matrix
are linearly independent over $\mathbb Z_p$.

If a system of equations over a group is $p$-nonsingular
for some prime number $p$
then it is also nonsingular,
as having a nontrivial combination
of integer rows over $\mathbb Z$
which is equal to zero implies that there is
a likewise combination in which
at least one coefficient is not divisible by $p$.
Particularly, unimodular systems of equations are nonsingular.
Nonsingularity of a finite system of equations
can be interpreted as existence of a non-zero
minor of rank $k$ in the matrix of
exponent sums of variables, where $k$ is the number of equations.
Similarly, $p$-nonsingularity can be seen
as existence of a likewise minor which is not divisible by $p$.
Hence, if a system of equations is finite, then
its nonsingularity implies its $p$-nonsingularity for some prime
number $p$.
This, however, is not true for infinite systems:
the system
$
\{x_p^p=1\},
$
where $p$ ranges over all prime numbers,
is nonsingular but is $p$-singular
for every prime number $p$.

Any nonsingular system of equations is solvable over
following types of groups:
\begin{itemize}
\item finite groups \cite{GR62};
\item locally residually finite groups
(follows from the previous item and some simple arguments);
\item hyperlinear groups (it is yet unknown if any group is such) \cite{P08};
\item locally indicable groups \cite{How81},
i.e. groups whose every nontrivial
finitely generated subgroup
admits an epimorphism onto the additive group $\mathbb Z$;
\item $p$-nonsingular systems are solvable over
locally $p$-indicable groups \cite{G83},
see also \cite{Kr85}.
A locally $p$-indicable group is a group whose every nontrivial
finitely generated subgroup
admits an epimorphism onto the additive group of $\mathbb Z_p$;
\item there is a conjecture (yet neither proved nor disproved)
that all nonsingular systems of equations
are solvable over every group \cite{How81}.
\end{itemize}

In fact, the solvability of a system
$\{w_i=1\}_{i\in I}$ of equations in variables $\{x_j\}_{j\in J}$ over a group $G$
is equivalent to the injectivity of the natural homomorphism
$$
G \to \left( G \ast F(X) \right)
/
\langle \langle
\{w_i\}_{i \in I}
\rangle \rangle,
$$
where $X=\{x_j\}_{j\in J}$.
This causes, in particular, the locality property:
the system of equations over a group is solvable
if and only if its
every finite subsystem is solvable.
That is why
it suffices to consider only finite systems to research solvability
of nonsingular equations over groups.

Consider a finite $\pi$-nonsingular system of equations over a group
and the matrix of the exponent sums of variables in these equations.
Suppose that there are $k$ equations and $n$ variables in the system and $k<n$
(i.e. the matrix of the system is not square).
Using transformations
(i.e. multiplying one equation by another
or by an inverse of another and rearranging equations)
of equations in the system
(which lead to transformations of rows
in the matrix)
and changes of variables
(which lead to transformations of columns in the matrix),
we can transform the matrix to a triangular 
matrix in which the first $k$ columns modulo $p$
are linearly independent over $\mathbb Z_p$
for each $p\in \pi$.
Therefore, this transformed system of equations
remains $\pi$-nonsingular
even after elimination of all variables except the first $k$ ones
(for example, the other variables can be changed to $1$).
Thus, any finite $\pi$-nonsingular system of equations
can be reduced to a likewise system
in which there are as many equations as there
are variables (i.e. the matrix of the system is square).
So it can be assumed for convenience
that finite nonsingular, $\pi$-nonsingular or unimodular
systems are ``square''.

One can also research solving
systems of equations in groups which are similar
to the initial group.
For example, one can try to solve a
system of equations over an abelian group
in an abelian group as well
(where the best case is being able to solve
the system in the initial group itself).
Let us start with finite systems.

If a group is finite, then every finite nonsingular system of equations
over this group is solvable in a finite group as well \cite{GR62}.
If a group is abelian, then every finite nonsingular system of equations over this group
is solvable in an abelian group, namely in the divisible hull of the group.
For example, one can solve the system by Cramer's rule
(having in mind that it is possible to divide
in a divisible abelian group, not necessarily uniquely).
Moreover, if a finite system of equations over an abelian group
is unimodular, then the system is solvable in this abelian group itself.
This is also true for all nilpotent groups:
it is a special case of Shmel'kin's theorem.
Before we formulate this theorem, let us remind a couple
of definitions.
\begin{defin}
Let $\pi$ be a set of some prime numbers.
A group $G$ is called {\it $\pi$-divisible}
if any element of $G$ is a $p$th power
of another element of $G$ for each prime $p \in \pi$.

Particularly, any group is $\varnothing$-divisible.
\end{defin}
\begin{defin}
Let $\pi$ be a set of some prime numbers.
A group $G$ is called {\it $\pi$-torsion free}
if it has no elements whose order is contained in $\pi$.

Particularly, any group is $\varnothing$-torsion free.
\end{defin}
Let us denote the complement of $\pi$ in the
set of all prime numbers by $\pi'$.
\begin{Shmelkintheorem*}[\!\!\cite{Sh67}]
Let $\pi$ be a set of some prime numbers.
Suppose that $G$ is a $\pi$-divisible
$\pi$-torsion-free
locally nilpotent group.
Then every square finite $\pi'$-nonsingular system of equations
over $G$ has a unique solution in $G$.
\end{Shmelkintheorem*}

It follows that every non-square finite $\pi'$-nonsingular system
of equations over $G$
has a solution (although not unique) in $G$
as well, as discussed above.

The case where a system of equations over a group
does not have solutions in groups from the same class, is also possible.
For instance, there is a unimodular equation over
a metabelian group such that the equation
is unsolvable in metabelian groups \cite{KMR24}.

The solution group of an infinite system of equations
can differ from solution groups of finite systems.
For example, an infinite nonsingular system of equations
over a finite group can be unsolvable in finite groups,
as shown by the system
$\{x_1^2=a,x_2^2=x_1,x_3^2 = x_2, \ldots \}$ over
$\langle a \rangle_2$, i.e. cyclic group of order $2$.
That is why for finding solutions of infinite
systems of equations in groups with ceratin properties
it is not enough to consider only finite subsystems.

The articles \cite{KMR24,M24} are concerned with
finding solutions of infinite systems of equations
over solvable groups in solvable groups as well.
\begin{theorem*}[\!\!\cite{KMR24}]
Suppose that a group $G$
has a subnormal series with abelian factors
$$
G=G_1 \triangleright G_2 \triangleright  \ldots \triangleright G_n \triangleright G_{n+1} =\{1\},
$$
in which all factors except the last one are torsion free.
Then any (including the infinite ones)
nonsingular system of equations over $G$
has a solution in some group $\tilde G$, which also
has a subnormal series with abelian factors
$$
\tilde G= \tilde G_1\triangleright  \tilde G_2
\triangleright \ldots \triangleright  \tilde G_{n}
\triangleright  \tilde G_{n+1}=\{1\},
$$
in which all factors except the last one are torsion free.
If $G_n$ is $\pi$-torsion free, then
$\tilde G$ can be chosen such that $\tilde G_n$
is $\pi$-torsion free as well.
\end{theorem*}
\begin{theorem*}[\!\!\cite{M24}]
Suppose that a group $G$
has a subnormal series with abelian factors
$$
G=G_1 \triangleright G_2 \triangleright  \ldots \triangleright G_n \triangleright G_{n+1} =\{1\},
$$
in which all factors except the last one are torsion free.
Then $G$ embeds into a group $\hat G$, which:
\begin{itemize}
\item has a subnormal series with abelian factors
$$
\hat G= \hat G_1\triangleright  \hat G_2
\triangleright \ldots \triangleright  \hat G_{n}
\triangleright  \hat G_{n+1}=\{1\},
$$
in which all factors except the last one are torsion free;
\item contains a solution of any nonsingular
system of equations over itself. Particularly,
$\hat G$ contains a solution of every such system
over $G$.
\end{itemize}
If $G_n$ is $\pi$-torsion free, then
$\hat G_n$ is $\pi$-torsion free as well.
\end{theorem*}
\begin{theorem*}[\!\!\cite{M24}]
Suppose that a group $G$
has a subnormal series with abelian factors
$$
G=G_1 \triangleright G_2 \triangleright  \ldots \triangleright G_n \triangleright G_{n+1} =\{1\},
$$
in which all factors except the last one are $p'$-torsion free.
Then $G$ embeds into a group $\hat G$, which:
\begin{itemize}
\item has a subnormal series with abelian factors
$$
\hat G= \hat G_1\triangleright  \hat G_2
\triangleright \ldots \triangleright  \hat G_{n}
\triangleright  \hat G_{n+1}=\{1\},
$$
in which all factors except the last one are $p'$-torsion free;
\item contains a solution of any $p$-nonsingular
system of equations over itself. Particularly,
$\hat G$ contains a solution of every such system
over $G$.
\end{itemize}
If $G_n$ is $\pi$-torsion free, then
$\hat G_n$ is $\pi$-torsion free as well.
\end{theorem*}

The present article studies solvability of
infinite systems of equations in abelian and nilpotent groups:
it is researched which abelian and nilpotent groups
contain solutions of all infinite unimodular
systems of equations over themselves
(having in mind that every nilpotent group
has a solution of every finite unimodular
system of equations over itself).
It turns out that not even all abelian groups
have this property. For instance, the group
$\mathbb Z_2 \oplus \mathbb Z_3 \oplus \mathbb Z_5 \oplus \ldots$
has no solutions of the system of equations (in the additive notation)
$$
\begin{cases}
x+ 2y_2= (1,0,0,\ldots)\\
x + 3y_3 = (0,1,0,\ldots)\\
x + 5y_5 = (0,0,1,\ldots)\\
\ldots
\end{cases}
$$

Indeed, suppose that the said group contains
a solution $\{\tilde x, \tilde y_2, \tilde y_3, \tilde y_5, \ldots\}$ of this system.
Then, as $\tilde x+ 2\tilde y_2= (1,0,0,\ldots)$
and the first coordinate of $2\tilde y_2$,
which is an element of $\mathbb Z_2$
multiplied by $2$,
equals $0$, we get that the first coordinate
of $\tilde x$ is not equal to $0$.
Also, since $\tilde x + 3\tilde y_3 = (0,1,0,\ldots)$
and the second coordinate of $3\tilde y_3$
is equal to $0$, as it is an element of $\mathbb Z_3$
multiplied by $3$, the second coordinate
of $\tilde x$ is not equal to $0$ as well.
Likewise we get that all the coordinates of $\tilde x$
are not equal to $0$. In other words, 
$\tilde x$ has an infinite number of coordinates not equal to $0$
and hence can not belong to $\mathbb Z_2 \oplus \mathbb Z_3 \oplus \mathbb Z_5 \oplus \ldots$.
Thus we arrive at a contradiction, so the group
$\mathbb Z_2 \oplus \mathbb Z_3 \oplus \mathbb Z_5 \oplus \ldots$
does not have a solution to the system above.

In Section \ref{AbCase}, we prove the following criterion.
\begin{theorem} \label{Criteria}
Suppose that $A$ is a periodic abelian group.
Then every unimodular system of equations over $A$ is solvable
in $A$ if and only if the reduced part of $A$ has bounded period.
\end{theorem}

Here the reduced part of an abelian group $A$
is any group $C$ such that $A = C \oplus D$
with $D$ divisible and $C$ reduced.
Such decomposition always exists, and in every
such decomposition $D$ is the largest divisible subgroup
of $A$ and $C$ is isomorphic to $A/D$,
so the reduced part of an abelian group
is unique up to isomorphism.

In Section \ref{NilpCase},
using this criterion and the methods of \cite{Sh67},
we prove the following result,
which is a partial generalization
of Shmel'kin's theorem to the infinite systems
of equations for nilpotent groups of bounded period.
\begin{theorem}\label{theorembound}
Suppose that $G$ is a nilpotent group of bounded period.
Then every (not even necessarily finite)
unimodular system of equations over $G$
is solvable in $G$ itself.
\end{theorem}
Also, in this section we prove the same generalization
of Shmel'kin's theorem for divisible nilpotent groups
(not even necessarily torsion free).
\begin{theorem}\label{theoremdiv}
Suppose that $G$ is a divisible nilpotent group.
Then every (not even necessarily finite)
nonsingular system of equations over $G$
is solvable in $G$ itself.
\end{theorem}

In Section \ref{OpenQuest}, we formulate some open questions
on solvability of infinite systems of equations in groups.

The author thanks the Theoretical Physics and Mathematics
Advancement Foundation ``BASIS''.
The author also thanks Anton Klyachko
for valuable remarks, the mentorship
and for the example from Question \ref{torsionfreebad}.
The author thanks an anonymous referee
for useful comments.

\section{Abelian groups} \label{AbCase}

In this section, we look for solutions of systems of equations
over an abelian group in an abelian group.
Hence, if we rearrange some coefficients and variables
in equations of a system, we get an equivalent
(for our purpose) system.
So, we can assume that variables commute with each other and with coefficients.
In this case we can use the additive notation
instead of the multiplicative one.
Thus, in this section we may assume that
each equation in a system has the form
$k_1x_1 + \ldots + k_nx_n = a$ for some $n\in \mathbb N$
(not necessarily the same $n$ for all equations),
where $a$ is the coefficient of the equation, $x_i$ are
variables and $k_i$ are integer numbers.
The other systems of equations are equivalent to these
(in the sense of searching for solutions in abelian groups).
In this case $k_i$ is the exponent sum of
$x_i$ in the equation.
In other words, the left hand sides of such equations
are the rows of exponent sums of variables.

Each abelian groups is a direct sum of a divisible
abelian group and a reduced abelian group.
Any nonsingular system of equations over a divisible
abelian group $A$ has a solution in $A$ itself,
as the following lemma shows.

\begin{lemma}\label{divisible}
Suppose that $A$ is a divisible abelian group.
Then every (including the infinite ones)
nonsingular system of equations over $A$
is solvable in $A$ itself.
\end{lemma}
\begin{proof}
Let $\{m_i = a_i\}_{i \in I}$ be
a nonsingular system of equations over $A$,
where $a_i\in A$, while $m_i$ is an element of
the $\mathbb Z$-module
$\sum_{j \in J} \mathbb Z \cdot x_j$,
which can be seen as an integer row. 

The natural mapping
$A \to \left( A\oplus \sum \mathbb Z \cdot x_j \right)
/
\langle \{
m_i - a_i
\} \rangle$
is an embedding because the rows $m_i$ are independent
as elements of $\sum \mathbb Z \cdot x_j$.
As $A$ is a divisible abelian group,
it is a direct summand of
$\left( A\oplus \sum \mathbb Z \cdot x_j \right)
/
\langle \{
m_i - a_i
\} \rangle$.
The images of the elements $x_j$
under the projection onto $A$
form the solution of
$\{m_i=a_i\}$
in $A$, as needed.
\end{proof}

That is why from now we consider only reduced abelian groups.

\subsection{Negative case}

Let us start with reduced abelian groups which
have no solutions of some unimodular systems of equations over themselves.

\begin{prop}\label{pbad}
Suppose that $A$ is a reduced abelian $p$-group.
Suppose that the period of $A$ is not bounded.
Then there is a unimodular system of equations over $A$,
which has no solutions in $A$.
\end{prop}
\begin{proof}
Let us recall that the height of an element $a$
in an abelian $p$-group is the maximal number $k$,
such that the equation $p^kx=a$
is solvable in this group, or infinity if there is no maximal number
among such numbers $k$.

Note that $A$ can be assumed to have no
elements of infinite height aside from $0$.
Indeed, if $A$ contains a solution of every unimodular system of equations
over itself, then this fact is also true for any quotient group of $A$.
In particular, it is true for the first Ulm factor $\bar A^0$ of $A$,
i.e. quotient group of $A$ by the group $A^1$
consisting of elements of $A$ which have infinite height.

If $A$ has no elements of infinite height (except $0$),
then $A^1=\{0\}$ and $A \cong \bar A^0$.
And if $A^1\neq \{0\}$, then note that in this case
$\bar A^0$ has unbounded period as well
and has no elements of infinite height aside from $0$ \cite[\S 27]{Kurosh67}.

Therefore, we can assume that
$A$ is an abelian $p$-group of unbounded period
and without non-zero elements of infinite height
(notice that the last property is stronger than the property of being reduced).
Now we show that for every natural number $N$
the group $A$ has a cyclic direct summand of order not less than $p^N$.

Indeed, take an arbitrary element
$a\in A$ of order $p^N$.
The group $A$ is separable, as it has no elements of infinite height
\cite[\S 65]{Fuchs73},
hence $a$ lies in a direct summand of
$A$ which is a direct sum of some cyclic groups.
Since the order of $a$ is $p^N$,
at least one of these cyclic groups has order not less than $p^N$.

$A$ has a cyclic direct summand
$\langle a_1 \rangle_{p^{k_1}}$
of order not less than $p$:
$$
A = \langle a_1 \rangle_{p^{k_1}} \oplus B_1.
$$

Now note that $B_1$ is also an abelian
$p$-group of unbounded period
without non-zero elements of infinite height.
It has a cyclic direct summand of order not less than $p^{2k_1+1}$.
We get the following decomposition of $A$:
$$
A = \langle a_1 \rangle_{p^{k_1}} \oplus \langle a_2 \rangle_{p^{k_2}} \oplus B_2,
$$
where $k_2 > 2k_1$. Similarly,
$B_2$ has a direct summand $\langle a_3 \rangle_{p^{k_3}}$
such that $k_3 > 2k_2$
and so on.

On the $n$th step we get the following decomposition of $A$:
$$
A = \langle a_1 \rangle_{p^{k_1}} \oplus \ldots \oplus \langle a_n \rangle_{p^{k_n}} \oplus B_n,
$$
where $k_{i+1} > 2 k_{i}$
for each $i$ from $1$ to $n$.

Consider the following system of equations over $A$:
$$
\{x_i -p^{k_{i} - k_{i-1}}x_{i+1} = a_i \mid i \in \mathbb N \},
$$
where $k_0=0$.
This system is unimodular, as can be seen from the matrix
of exponent sums of the variables in the equations:
$$
\begin{pmatrix}
1 & -p^{k_1} & 0 & \dots \\
0 & 1 & -p^{k_2-k_1} & \dots \\
0 & 0 & 1 & \dots \\
\dots & \dots & \dots & \dots \\
\end{pmatrix}
$$

Suppose that this system has a solution $\{\tilde x_i\}$.
Then the following equalities hold:
$$
\tilde x_1 - p^{k_i}\tilde x_{i+1}= a_1 + p^{k_1}a_2 + \ldots+ p^{k_{i-1}}a_{i}.
$$
Take a look at the $i$th of these equations in the following decomposition of $A$:
$$
A = \langle a_1 \rangle_{p^{k_1}} \oplus \ldots \oplus \langle a_i \rangle_{p^{k_i}} \oplus B_i.
$$
In the component $\langle a_i \rangle_{p^{k_i}}$ the element $p^{k_i}\tilde x_{i+1}$
is equal to zero while $p^{k_{i-1}}a_{i}$ has order
$p^{k_i - k_{i-1}} > p^{k_{i-1}}$.
So for any $j \in \mathbb N$ the order of $\tilde x_1$
is not less than $p^{k_j}$.
Now note that
$$
k_j > 2k_{j-1} > \ldots > 2^{j-1}k_1 \geqslant 2^{j-1}.
$$
It means that the order of $\tilde x_1$ is not less than $p^{2^{j-1}}$ for every $j$.
It is possible only if $\tilde x_1$ has infinite order,
but $A$ has no elements of infinite order, so
the assumption of having a solution of the system in $A$
is false.

As a result, we have a unimodular system of equations
over $A$ which has no solutions in $A$, as needed.
\end{proof}

Now let us move from $p$-groups to the
other groups of unbounded period.

\begin{lemma}\label{bad}
Suppose that $A$ is a reduced periodic abelian group
of unbounded period.
Then there is a unimodular system of equations
over $A$
which has no solutions in $A$.
\end{lemma}
\begin{proof}
The periodic abelian group $A$ is a direct sum
of its $p$-components:
$$
A = \bigoplus_{i \in I} A_{p_i}.
$$
First consider the case where $I$ is finite.
Then some component $A_p$ has unbounded period
(otherwise the period of $A$ is bounded).
If any unimodular system of equations over $A$ has a solution in $A$,
then this is true for any quotient group of $A$.
However, Proposition \ref{pbad} shows that
for $A_p$ this fact is not true.
So, not all unimodular systems of equations over $A$
are solvable in $A$. One can even give an example
of a system which has no solutions in $A$: if the projection
of the system onto $A_p$ is the system from the proof of Proposition \ref{pbad},
then the system itself is unsolvable in $A$.

Now suppose that $I$ is infinite (in this case we can assume $I = \mathbb N$,
although the numbers $p_i$ do not have to range over all prime numbers).
In each $A_{p_i}$ pick an element
$a_i \in A_{p_i}\setminus p_iA_{p_i}$, which exists
since the abelian $p_i$-group $A_{p_i}$ is not divisible.
Consider the following system of equations:
$$
\{x + p_i y_i = a_i \mid i \in \mathbb N \}.
$$
It has the following matrix of exponent sums:
$$
\begin{pmatrix}
1 & p_1 & 0 & 0 &\dots \\
1 & 0 & p_2 & 0 &\dots \\
1 & 0 & 0 & p_3 & \dots \\
\dots & \dots & \dots & \dots & \dots \\
\end{pmatrix}
$$
Clearly, the system is $p_i$-nonsingular
for every $p_i$ and $p$-nonsingular
for the remaining prime numbers $p$, so this system is unimodular.

Suppose that this system has a solution $\{\tilde x, \tilde y_1, \tilde y_2, \ldots\}$
in $A$. It means that this solution lies in the Cartesian product
$\prod\limits_{i \in \mathbb N} A_{p_i}$ of $p$-components.
Look at the component $A_{p_i}$ and the equation $\tilde x + p_i \tilde y_i = a_i$.
As $a_i \notin p_i A_{p_i}$, $p_i$-component of $\tilde x$
is not $0$ (otherwise $a_i$ is equal to $p_i \hat y_i \in p_iA_{p_i}$, where
$\hat y_i$ is the $p_i$-component of $\tilde y_i$).
So we get that every component of $\tilde x\in\prod\limits_{i \in \mathbb N} A_{p_i}$
is not equal to $0$.
As there is an infinite number of such components in $\prod\limits_{i \in \mathbb N} A_{p_i}$,
$\tilde x$ does not lie in the direct sum of the components, which means $\tilde x \notin A$.
We get a contradiction, so the assumption of finding a solution
of the system in $A$ is false.
\end{proof}

\subsection{Positive case}

Now we look at groups which contain solutions
to all unimodular systems of equations over themselves.

\begin{prop}\label{pvecgood}
Any 
$p$-nonsingular system of equations
over an abelian group $A$
of prime period $p$ has a solution in $A$ itself.
\end{prop}

\begin{proof}
Consider a $p$-nonsingular system
$\{m_i = a_i\}_{i \in I}$ of equations in variables
$\{x_j\}_{j \in J}$ over $A$,
where $m_i$ is an integer linear combination of the variables $\{x_j\}$
(so $m_i$ is an element of the free $\mathbb Z$-module
$\sum\limits_{j \in J} \mathbb Z \cdot x_j$),
and $a_i$ is an element of $A$.

Note that the abelian group $A$ of period $p$ is a
vector space over $\mathbb Z_p$
(and a $\mathbb Z$-module as well),
whereas a solution of the system $\{m_i=a_i\}$ in $A$ is
a homomorphism of $\mathbb Z$-modules
$f \colon \sum\limits_{j \in J} \mathbb Z \cdot x_j \to A$
such that $f(m_i)=a_i$.

Any such homomorphism passes through the $\mathbb Z$-module homomorphism
$r_p\colon\sum\limits_{j \in J} \mathbb Z \cdot x_j \to \sum\limits_{j \in J}\mathbb Z_p \cdot x_j$,
under which integer coefficients are taken modulo $p$.
Therefore, we need to find a linear map
$\tilde f \colon \sum\mathbb Z_p \cdot x_j \to A$
such that $\tilde f (\tilde m_j) = a_j$,
where $\tilde m_j=r_p(m_j)$.

Now we construct the mapping $\tilde f$. Note that vectors
$\{\tilde m_i\}$ are linearly independent over $\mathbb Z_p$ by the hypotheses of the proposition
(as $m_i$ are the exponent sums rows of variables
$\{x_j\}$ in $w_i$, and these rows are linearly independent
modulo $p$), that is why there is a set
$\{v_k\}_{k \in K}$ of vectors
such that $\{\tilde m_i,v_k\}$ is a basis of
$\sum\mathbb Z_p \cdot x_j$.
Now define $\tilde f$ on
$\{\tilde m_i\}\cup \{v_k\}$ so that
$\tilde f(\tilde m_i)=a_i$ and $\tilde f(v_k)$ are arbitrary
(e.g. $\tilde f(v_k)=0$).
We can now define $\tilde f$ on $\sum\mathbb Z_p \cdot x_j$
as a linear map, and this is the desired linear map.

So, we get a solution of $\{m_i=a_i\}$ in $A$.
\end{proof}

\begin{prop}\label{pgood}
Any $p$-nonsingular system of equations
over an abelian $p$-group $A$ of bounded period
is solvable in $A$ itself.
\end{prop}
\begin{proof}
Let the period of $A$ be $p^n$.
Use induction on $n$.

The base case ($n=1$) is Proposition \ref{pvecgood}.

Now show that if the statement holds
for $n=l$, then it holds for $n=l+1$.
Consider a $p$-nonsingular system
$\{m_i = a_i\}_{i \in I}$ of equations in variables
$\{x_j\}_{j \in J}$ over $A$.

Look at the quotient group $A/pA$.
Take the induced system $\{m_i=\overline{a}_i\}_{i\in I}$
of equations over $A/pA$, where $\overline{a}_i$ is the image of $a_i$
under taking the quotient $A \to A/pA$.
It is a $p$-nonsingular system of equations over the abelian group $A/pA$
of period $p$, therefore, by Proposition \ref{pvecgood}
the system has a solution $\{\bar c_j\}_{j\in J}$ in $A/pA$.

For each $\bar c_j \in A/pA$ pick an element
$c_j \in A$ of its preimage.
Then after the substitution
$x_j \mapsto x_j+c_j$ the system
$\{m_i = a_i\}$ transforms into
$\{m_i = b_i\}$, where $b_i \in pA$,
since the induced system $\{m_i=\overline{a}_i\}$
over $A/pA$
transforms into
$\{m_i=1\}$.

The system $\{m_i = b_i\}$ can be seen as a system of equations over
$pA$, which is an abelian group of period $p^l$.
This system has a solution in $pA$ by induction hypothesis.
This solution is also a solution of the system in $A \supset pA$,
as needed.
\end{proof}

Now we can move on to arbitrary groups of bounded period.

\begin{lemma}\label{good}
Suppose that $A$ is an abelian group of bounded period.
Then every unimodular system of equations over $A$
is solvable in $A$ itself.
\end{lemma}
\begin{proof}
Suppose that $p_1, p_2,\ldots, p_k$
are all the prime divisors of the period of $A$.
Then the following decomposition holds:
$$
A = A_{p_1}
\oplus A_{p_2} \oplus \ldots \oplus A_{p_k},
$$
where $A_{p_i}$ is the $p_i$-component of $A$
(which has bounded period $p_i^{n_i}$).
For the components $A_{p_i}$
the statement of the lemma follows
from Proposition \ref{pgood},
therefore, it is also true for their direct sum.
\end{proof}

\begin{rem}
In fact, the system does not need to be precisely unimodular
as long as it is $\{p_1,\ldots,p_k\}$-nonsingular
(since it suffices for system to be $p_i$-nonsingular
for each prime $p_i$ which divides the period of the group).
\end{rem}
\begin{rem}
In particular, a unimodular (not necessarily finite) system of equations
over a finite abelian group is always solvable in that group itself.
This is also true for $\{p_1,p_2,\ldots,p_k\}$-nonsingular systems,
where $p_1,\ldots,p_k$ are all the prime divisors of the order of the abelian group.
\end{rem}

\subsection{The final criterion}

Now we can prove the criterion formulated in the introduction.

\begin{proof}[Proof of Theorem \ref{Criteria}]
Decompose $A$ into a direct sum of a divisible abelian group
and a reduced one.
The divisible abelian group has a solution of every
nonsingular system of equations over itself (by Lemma \ref{divisible}).
Therefore, nonsingular (particularly, unimodular)
systems of equations have a solution in $A$
if and only if
the same is true for its reduced part.
Hence we can assume $A$ to be reduced.

By Lemma \ref{bad}, if any unimodular system of equations
over $A$ is solvable in $A$ itself, then $A$ must have bounded period.

And if $A$ has bounded period, then any unimodular system of equations over $A$
has a solution in $A$ by Lemma \ref{good}.
\end{proof}
\begin{rem}
An analogous criterion can be formulated for abelian $\pi$-groups
(a $\pi$-group is a group, elements of which have finite orders whose
every prime divisor lies in $\pi$):
\begin{quote}
Suppose that $A$ is an abelian $\pi$-group, where
$\pi$ is a set of some prime numbers.
Then every $\pi$-nonsingular system of equations over
$A$ is solvable in $A$
if and only if
the reduced part of
$A$ has bounded period.
\end{quote}
\end{rem}

\section{Nilpotent groups} \label{NilpCase}

Now we switch back to the multiplicative notation.
The results of Section \ref{AbCase} can be used to research
solvability of systems of equations in nilpotent groups.

\subsection{Periodic nilpotent groups}

\begin{proof}[Proof of Theorem \ref{theorembound}]
Denote the nilpotency class of $G$ by $s$.
We argue by induction on $s$.

The base case ($s=1$) follows from Theorem \ref{Criteria}.

Now we show that if the statement is true for
$s=l$,
then it holds for $s=l+1$ as well.
Consider a unimodular system $\{w_i=1\}_{i\in I}$
of equations in variables $\{x_j\}_{j \in J}$ over $G$.

Take a look at the system $\{\overline{w}_i=1\}_{i\in I}$ over $G/Z(G)$,
where for each $i \in I$ the word $\overline{w}_i$ is
the word $w_i$ with coefficients changed to their images
under taking the quotient $G \to G/Z(G)$.
The system $\{\overline{w}_i=1\}$ is a unimodular system of equations
over $G/Z(G)$ (which is a nilpotent group of class $l$
and of bounded period),
therefore, the system is solvable in the quotient group
$G/Z(G)$ itself by induction hypothesis.

This means that we can make a substitution of the variables
in the system $\{w_i=1\}$ over $G$
such that for each $i \in I$ product $b_i$ of coefficients of the word $w_i$
lies in $Z(G)$
(i.e. the image of $b_i$ under taking the quotient $G \to G/Z(G)$ is $1$).
So we can assume that each word $w_i$ has this property initially
(that is, for every $w_i$, the product of the coefficients of $w_i$ lies in $Z(G)$).

Search the solution $\{\tilde x_j\}$ among the elements of $Z(G)$.
Finding a solution of $\{w_i=1\}$, in which all $\tilde x_j$ are
from $Z(G)$, is the same as finding a solution of $\{v_ib_i=1\}_{i \in I}$
in $Z(G)$, where for each $i \in I$ the word $v_i$ is $w_i$ with coefficients removed.
The latter system can be seen as a system of equations over $Z(G)$,
thus (since the system is unimodular and $Z(G)$ has bounded period)
it has a solution in $Z(G)$ by Theorem \ref{Criteria}.
This solution is also a solution of $\{w_i=1\}$, as needed.
\end{proof}

\begin{rem}
Suppose that all the prime divisors of the period
of the nilpotent group $G$ are
$p_1,p_2,\ldots,p_k$. Then any
$\{p_1,p_2,\ldots,p_k\}$-nonsingular system of equations
over $G$ is solvable in $G$ as well.

For abelian groups it follows from the fact that
the periods of primary components
in the corresponding decomposition of such group
can only be powers of prime numbers from $\{p_1,p_2,\ldots,p_n\}$.

For nilpotent groups we use the same argument as in the proof above.

In particular, an infinite-system analogue of Shmel'kin's
theorem (more precisely, of its existence part)
holds for finite $p$-groups.
\end{rem} 

For finite nilpotent groups we can prove solvability
of infinite unimodular systems of equations
in another way:
\begin{proof}[Alternative proof]
Let $V$ be a unimodular system of equations over a finite nilpotent group $G$.
For any finite subsystem $W\subset V$ consider the copy
$G_W$ of $G$, which, by Shmel'kin's theorem,
has a solution $\{x_{W,j}\}_{j \in J}$ of $W$.

Now look at the set $M$ of all finite subsystems of $V$.
Consider the following family of sets:
by $S_W$ denote a set of all such finite systems
$\tilde W \subset V$ that $W \subset \tilde W$.
Note that $\{S_W\}$ has the finite intersection property.
Indeed, the finite intersection
$S_{W_1} \cap \ldots \cap S_{W_n}$
contains non-empty set $S_{W_1 \cup \ldots \cup W_n}$
as a subset.
Therefore, we can complete $\{S_W\}$ to an ultrafilter $\mathcal U$.

Now consider the ultraproduct
$\left( \prod\limits_{W \in M} G_W \right) \Bigl/ \mathcal U$
(in other words, the ultrapower of $G$). Embed $G$ into this ultrapower diagonally.
The set $\{ \tilde x_j\}_{j\in J}$, where $\tilde x_j$ is the element
of the ultraproduct
whose $G_W$-component is $x_{W,j}$ for each $W \in M$,
forms a solution of $V$.
Indeed, denote by $\tilde w_i$ the word $w_i$ after the substitution of variables by
the elements $\{\tilde x_j\}$ of the ultraproduct.
Then the set of indices of coordinates of $\tilde w_i$, which are equal to $1$,
contains $S_{\{w_i=1\}}$ as a subset, and so belongs the the ultrafilter.
Hence $\tilde w_i=1$ in the ultraproduct itself, and this holds for each $i \in I$.

But this ultraproduct is identical to $G$, as the order of an
ultrapower of a finite group $G$ is not greater than $|G|$
\cite[chapter IV, \S 8, item 8.5]{Mal70}.
So the infinite unimodular system $V$
has a solution in $G$ itself.
\end{proof}

\subsection{Divisible nilpotent groups}

Before we prove the result concerning divisible nilpotent groups,
let us prove the following lemma.

\begin{lemma}\label{divisiblenilp}
Suppose that $G$ is a divisible nilpotent group
of class $s$.
Then for every $i\in \{1,\ldots,s\}$ the term
$Z_i(G)$ of upper central series of $G$
is divisible as well.
\end{lemma}
\begin{proof}
Let us adopt the following convention:
the commutator $[g,h]$ denotes $g^{-1}h^{-1}gh$.

Let $z$ be an element of $Z_i(G)$ and $n$ be a natural number.

Since $G$ is divisible, there is such an element $w\in G$ that $w^n=z$
Now we prove that $w \in Z_i(G)$.
Let $g \in G$. There is an element $h \in G$ such that $h^n = g$.

Consider the commutator $[w,g]=[w,h^n]$.
Assume that it is already shown that $w \in Z_j(G)$.
This is true at least for $j=s$, as $Z_s (G) = G$.
Using the commutator indentities
$$
[a,bc]=[a,c][a,b][[a,b],c];\quad [ab,c]=[a,c][[a,c],b][b,c],
$$
we get that
$$
[w,h^n] = [w,h^{n-1}] [w,h] [[w,h],h^{n-1}].
$$
Now notice that as $w \in Z_j(G)$,
$[w,h_1] \in Z_{j-1}(G)$ and $[[w,h_1],h_2] \in Z_{j-2}(G)$
for all $h_1,h_2 \in G$.
This means that $[w,h^n] = [w,h^{n-1}][w,h]$ modulo $Z_{j-2}(G)$.
Dealing with $[w,h^{n-1}]$ in the same way and so on,
we receive that $[w,h^n] = [w,h]^n$ modulo $Z_{j-2}(G)$.

Likewise we get that $[w^n,h]=[w,h]^n$
modulo $Z_{j-2}(G)$.
So,
$[w,g] = [w,h^n] = [w,h]^n = [w^n,h] = [z,h]$
modulo $Z_{j-2}(G)$.
As $z \in Z_{i}(G)$,
$[z,h] \in Z_{i-1}(G)$.
Hence,
$[w,g] \in Z_{i-1}(G)Z_{j-2}(G)$.
If $j>i$, then
$[w,g] \in Z_{j-2}(G)$ for any $g \in G$.
So the element $w$ itself lies in $Z_{j-1}(G)$
(since it commutes modulo $Z_{j-2}(G)$
with any element of $G$).

As a result, we proved that if $w \in Z_j(G)$ and $j>i$,
then $w \in Z_{j-1}(G)$.
It follows that \\
$w \in Z_{s-1}(G),w \in Z_{s-2}(G),\ldots,
w \in Z_i(G)$, hence the statement of the lemma.
\end{proof}

\begin{rem}
For instance, the center of a divisible nilpotent group is divisible as well.
\end{rem}

\begin{rem}
In fact, we even prove that all terms of the upper central series
of a divisible nilpotent group $G$ are isolated subgroups
(i.e. any $n$th root, where $n$ is arbitrary, of any element of $Z_i(G)$
lies in $Z_i(G)$ as well).
\end{rem}

\begin{rem}
Notice that for $n$th root extraction in $Z_i(G)$
we needed only $n$th root extraction in $G$.
Therefore, the analogous fact for $\pi$-divisibility
(particularly, for $p$-divisibility) also holds:
\begin{quote}
Suppose that $G$ is a $\pi$-divisible nilpotent group.
Then each of its upper central ceries terms is
$\pi$-divisible as well.
\end{quote}
\end{rem}

\begin{proof}[Proof of Theorem \ref{theoremdiv}]
Denote the nilpotency class of $G$ by $s$.
We argue by induction on $s$.

The base case ($s=1$) follows from Lemma \ref{divisible}.

Now we show that if the statement is true for
$s=l$,
then it holds for $s=l+1$ as well.
Consider a nonsingular system $\{w_i=1\}_{i\in I}$
of equations in variables $\{x_j\}_{j \in J}$ over $G$.

Take a look at the system $\{\overline{w}_i=1\}_{i\in I}$ over $G/Z(G)$,
where for each $i \in I$ the word $\overline{w}_i$ is
the word $w_i$ with coefficients changed to their images
under taking the quotient $G \to G/Z(G)$.
The system $\{\overline{w}_i=1\}$ is a nonsingular system of equations
over $G/Z(G)$ (which is a divisible nilpotent group of class $l$),
therefore, the system is solvable in the quotient group
$G/Z(G)$ itself by induction hypothesis.

This means that we can make a substitution of the variables
in the system $\{w_i=1\}$ over $G$
such that for each $i \in I$ the product $b_i$ of coefficients of the word $w_i$
lies in $Z(G)$
(i.e. the image of $b_i$ under taking the quotient $G \to G/Z(G)$ is $1$).
So we can assume that each word $w_i$ has this property initially
(that is, for every $w_i$, the product of the coefficients of $w_i$ lies in $Z(G)$).

Search the solution $\{\tilde x_j\}$ among the elements of $Z(G)$.
Finding a solution of $\{w_i=1\}$, in which all $\tilde x_j$ are
from $Z(G)$, is the same as finding a solution of $\{v_ib_i=1\}_{i \in I}$
in $Z(G)$, where for each $i \in I$ the word $v_i$ is $w_i$ with coefficients removed.
The latter system can be seen as a system of equations over $Z(G)$,
thus (since the system is nonsingular and $Z(G)$ is divisible)
it has a solution in $Z(G)$.
This solution is also a solution of $\{w_i=1\}$, as needed.
\end{proof}

\section{Open questions} \label{OpenQuest}

\begin{question} \label{torsionfreebad}
Which torsion-free abelian groups contain solutions
of every unimodular system of equations over itself?
The group $\mathbb Z$ fails to have this property, as the following
system of equations (in the additive notation) shows.
$$
\begin{cases}
\ldots,\\
2y_2 - y_1 = 0,\\
2y_1 - x = 1,\\
x - 3z_1 = 0,\\
z_1 - 3z_2 = 0,\\
\ldots
\end{cases}
$$
\end{question}

\begin{question}
The same question for mixed abelian groups, i.e.
abelian groups which have
both elements of infinite order and
nontrivial elements of finite order.
\end{question}

\begin{question}
Can a criterion similar to Theorem
\ref{Criteria} be derived for nilpotent groups?
\end{question}


\begin{thebibliography}{30}

\bibitem{ABA21}
M. F. Anwar, M. Bibi, M. S. Akram,
On solvability of certain equations of arbitrary length
over torsion-free groups,
Glasgow Mathematical Journal, 63:3 (2021), 651--659.
See also arXiv:1903.06503

\bibitem{BE18}
M. Bibi, M. Edjvet,
Solving equations of length seven over torsion-free groups,
Journal of Group Theory, 21:1 (2018), 147--164.

\bibitem{B80}
S. D. Brodskii,
Equations over groups and groups with a single defining relation,
Russian Mathematical Surveys,
35:4
(1980),
165--165.

\bibitem{B84}
S. D. Brodskii,
Equations over groups, and groups with one defining relation;
Siberian Mathematical Journal,
25 (1984), 235--251.

\bibitem{Fuchs73}
L. Fuchs,
Infinite abelian groups,
Volume 2,
Academic Press,
New York and London
(1973)
\if 0
Russian version:
Л. Фукс,
Бесконечные абелевы группы,
том 2,
Издательство <<Мир>>,
Москва
(1977)
\fi


\bibitem{EH91}
M. Edjvet, J. Howie,
The solution of length four equations over groups,
{Trans. Amer. Math.  Soc.}, {326:1} (1991), 345--369.

\bibitem{EH21}
M. Edjvet, J. Howie,
On singular equations over torsion-free groups,
International Journal of Algebra and Computation, 31:3 (2021), 551--580.
See also arXiv:2001.07634

\bibitem{EdJu00}
M. Edjvet, A. Juh\'asz,
Equations of length 4 and one-relator products,
{Math. Proc. Cambridge Phil. Soc.},
129:2 (2000), 217--230.

\bibitem{G83}
S. M. Gersten,
Conservative groups, indicability and a conjecture of Howie,
J. Pure Appl. Algebra
29 (1983), 59--74

\bibitem{GR62}
M. Gerstenhaber, O. S. Rothaus,
The solution of sets of equations in groups,
{Proc. Nat. Acad. Sci. USA},
{48:9} (1962), 1531--1533.

\bibitem{How81}
J. Howie,
On pairs of 2-complexes and systems of equations over groups,
{J. Reine Angew Math.},
1981:324 (1981), 165--174.

\bibitem{IK00}
S. V. Ivanov, A. A. Klyachko,
Solving equations of length at most six over torsion-free groups,
Journal of Group Theory,
3:3 (2000), 329--337.

\bibitem{K93}
A. A. Klyachko,
A funny property of sphere and equations over groups,
Communications in Algebra,
21:7 (1993), 2555--2575.

\bibitem{K06}
A. A. Klyachko
How to generalize known results on equations over groups,
Math. Notes,
79:3(2006),
377--386.
See also arXiv:math.GR/0406382

\bibitem{KM23}
A. A. Klyachko, M. A. Mikheenko,
Yet another Freiheitssatz: Mating finite groups with locally indicable ones,
Glasgow Mathematical Journal,
65:2 (2023), 337--344.
See also arXiv:2204.01122

\bibitem{KMR24}
A. A. Klyachko, M. A. Mikheenko, V. A. Roman'kov,
Equations over solvable groups,
Journal of Algebra,
638 (2024),
739--750.
See also arXiv:2303.13240

\bibitem{KP95}
A. A. Klyachko, M. I. Prishchepov,
The descent method for equations over groups,
Moscow Univ. Math. Bull.,
50:4(1995),
56--58.

\bibitem{KT17}
A. Klyachko, A. Thom,
New topological methods to solve equations over
groups, Algebr. Geom. Topol.
17 (2017), no. 1, 331--353.
See also arXiv:1509.01376

\bibitem{Kr85}
S. Krsti\'c, 
Systems of equations over locally $p$-indicable groups,
Inventiones mathematicae,
81 (1985), 373--378.

\bibitem{Kurosh67}
A. G. Kurosh
The theory of groups,
second English edition,
Volume 1
(N. Y.,
Chelsea publishing company,
1960,
272 p.)
\if 0
Russian version
А. Г. Курош,
Теория групп,
издание 3
(М.:
Наука,
1967. --
648 с.)
\fi

\bibitem{Le62}
F. Levin,
Solutions of equations over groups,
{Bull. Amer. Math. Soc.},
{68:6} (1962), 603--604.

\bibitem{Mal70}
A. I. Mal'cev,
Algebraic systems
(Berlin, New York,
Springer-Verlag,
1973 xii,
317 p.)
\if 0
Russian version:
A. I. Mal'cev,
Algebraic systems
(M.:
Nauka
1970. --
392 p.)
\fi

\bibitem{M24}
M. A. Mikheenko,
$p$-Nonsingular systems of equations
over solvable groups,
Sbornik: Mathematics,
215:6
(2024),
775--789
\if 0
М. А. Михеенко,
О $p$-невырожденных системах уравнений
над разрешимыми группами,
Мат. сборник,
215:6
(2024),
61--76
\fi

\bibitem{NT22}
M. Nitsche, A. Thom,
Universal solvability of group equations,
Journal of Group Theory,
25:1 (2022), 1--10.
See also arXiv:1811.07737

\bibitem{P08}
V. G. Pestov,
Hyperlinear and sofic groups: A brief guide,
Bull. Symb. Log., 14:4 (2008), 449--480.
See also arXiv:0804.3968

\bibitem{Ro12}
V. A. Roman'kov,
Equations over groups,
Groups --- Complexity --- Cryptology,
4:2 (2012), 191--239.

\bibitem{Sh67}
A. L. Shmel'kin,
Complete nilpotent groups (in Russian),
Algebra i logika Seminar,
6:2
(1967),
111--114

\bibitem{Sh81}
H. Short,
Topological methods in group theory: the adjunction problem.
Ph.D. Thesis,
University of Warwick, 1981.

\bibitem{T18}
A. Thom,
Finitary approximations of groups and their applications,
Proceedings of the International Congress of Mathematicians ---
Rio de Janeiro 2018. Vol. III.
Invited lectures, World Scientific, Hackensack
(2018), 1779--1799.
See also arXiv:1712.01052
\end{thebibliography}
\end{document}